\documentclass[11pt]{article}
\usepackage{amsfonts}
\usepackage[all]{xy}
\usepackage{amsmath}
\usepackage{amsthm}
\usepackage{amssymb}
\usepackage{amscd}
\usepackage{graphicx}
\usepackage{microtype}
\usepackage{enumerate}
\usepackage{fullpage}
\usepackage{pinlabel}
\usepackage[top=1.3in,bottom=1.7in,left=1.3in,right=1.3in]{geometry}

\theoremstyle{definition}
\newtheorem{theorem}{Theorem}[section]

\newtheorem{corollary}[theorem]{Corollary}
\newtheorem{non-example}[theorem]{Non-example}

\newtheorem{lemma}[theorem]{Lemma}

\newtheorem{problem}[theorem]{Problem}

\newtheorem{remark}[theorem]{Remark}
\newtheorem{question}[theorem]{Question}

\newtheorem*{mainthm}{Theorem \ref{main thm}}
\newtheorem*{torfreethm}{Theorem \ref{tor free thm}}

\newcommand{\R}{\mathbb{R}}
\newcommand{\C}{\mathbb{C}}

\newcommand{\Z}{\mathbb{Z}}
\newcommand{\del}{\partial}
\newcommand{\boldhead}[1]{ {\bigskip \noindent \large \bfseries #1 \smallskip \\ }}

\DeclareMathOperator{\per}{per}
\DeclareMathOperator{\rel}{rel}

\DeclareMathOperator{\PSL}{PSL}
\DeclareMathOperator{\SO}{SO}
\DeclareMathOperator{\SL}{SL}
\DeclareMathOperator{\GL}{GL}
\DeclareMathOperator{\fix}{fix}
\DeclareMathOperator{\Diff}{Diff}
\DeclareMathOperator{\Homeo}{Homeo}

\DeclareMathOperator{\supp}{supp}

\title{Diffeomorphism groups of balls and spheres}
\author{Kathryn Mann}
\date{}

\begin{document}

\maketitle

\abstract{In this paper we discuss the relationship between groups of diffeomorphisms of spheres and balls.  We survey results of a topological nature and then address the relationship \emph{as abstract} (discrete) \emph{groups}.  
We prove that the identity component $\Diff^\infty_0(S^{2n-1})$ of the group of $C^\infty$ diffeomorphisms of $S^{2n+1}$ admits no nontrivial homomorphisms to the group of $C^1$ diffeomorphisms of the ball $B^m$ for any $n$ and $m$.  This result generalizes theorems of Ghys and Herman.  

We also examine finitely generated subgroups of $\Diff^\infty_0(S^n)$ and produce an example of a finitely generated torsion free group $\Gamma$ with an action on $S^1$ by $C^\infty$ diffeomorphisms that does not extend to a $C^1$ action of $\Gamma$ on $B^2$. 
}

\section{Introduction}

Let $M$ be a manifold and let $\Diff^r_0(M)$ denote the group of isotopically trivial $C^r$-diffeomorphisms of $M$.  If $M$ has boundary $\del M$, there is a natural map 
$$\pi: \Diff^r_0(M) \to \Diff^r_0(\del M)$$ given by restricting the domain of a diffeomorphism to the boundary.  The map $\pi$ is surjective, as any isotopically trivial diffeomorphism $f$ of the boundary can be extended to a diffeomorphism $F$ of $M$ supported on a collar neighborhood $N \cong \del M \times I$ of $\del M$ by taking a smooth isotopy $f_t$ from $f$ to the identity, and defining $F$ to agree with $f_t$ on $\del M \times \{t\}$.   

One way to measure the difference between the groups $\Diff^r_0(M)$ and $\Diff^r_0(\del M)$ is to ask whether $\pi$ admits a section.  By section, we mean a map $$\phi: \Diff^r_0(\del M) \to \Diff^r_0(M)$$ such that $\pi \circ \phi$ is the identity on $\Diff^r_0(\del M)$.  There are several categories in which to ask this, namely 
\begin{enumerate}[i)]
\item \textbf{Topological}: Require $\phi$ to be continuous, ignoring the group structure.
\item \textbf{(Purely) group-theoretic}: Only require $\phi$ to be a group homomorphism, ignoring the topological structure on $\Diff^r_0(M)$.
\item \textbf{Extensions of group actions}: In the case where no group-theoretic section exists, we ask the following \emph{local} (in the sense of group theory) question.  For which finitely generated groups $\Gamma$ and homomorphisms $\rho: \Gamma \to \Diff^r_0(\del M)$ does there exist a homomorphism $\phi: \Gamma \to \Diff^r_0(M)$ such that  $\pi \circ \phi = \rho$?   If such a homomorphism exists, we say that $\phi$ \emph{extends} the action of $\Gamma$ on $\del M$ to a $C^r$ action on $M$.  
\end{enumerate}

In this paper, we treat the case of the ball $M = B^{n+1}$ with boundary $S^n$.   Note in the category of \emph{homeomorphisms} rather than diffeomorphisms, there is a natural way to extend homeomorphisms of $S^{n}$ to homeomorphisms of $B^{n+1}$.  This is by ``coning off" the sphere to the ball and extending each homeomorphism to be constant along rays. The result is a continuous group homomorphism $$\phi: \Homeo_0(S^n) \to \Homeo_0(B^{n+1})$$ which is also a section of $\pi: \Homeo_0(B^{n+1}) \to \Homeo_0(S^{n})$ in the sense above. 
We will see, however, that the question of sections for groups of \emph{diffeomorphisms} is much more interesting!

\boldhead{Summary of results}
Our goal in this work is to paint a relatively complete picture of known and new results for the ball $B^n$.  Here is an outline.  

\begin{list}{}{\setlength\leftmargin{0in}
\setlength\parsep{0in}
\setlength\itemsep{.15in}
\setlength\listparindent{.2in}}
\item\textbf{Topological sections.}
In Section \ref{top sec} we give brief survey of known results on existence and nonexistence of topological sections, and the relationship between topological sections and exotic spheres.  The reader may skip this section if desired; it stands independent from the rest of this paper.  

\item \textbf{Group-theoretic sections.}
In contrast with the topological case, it is a theorem of Ghys that no group theoretic sections $\phi: \Diff^r_0(S^n) \to \Diff^r_0(B^{n+1})$ exist for any $n$ or $r$.   A close reading of Ghys' work in \cite{Gh} produces \emph{finitely generated} subgroups of $\Diff_0(S^{2n-1})$ that fail to extend to $\Diff_0(B^{2n})$ and we give an explicit presentation of such a group in Section \ref{group sec}.  These examples rely heavily on the dynamics of finite order diffeomorphisms.  

\item \textbf{Extending actions of torsion free groups.}
Building on Ghys' work and using results of Franks and Handel involving \emph{distortion elements} in finite groups, in Section \ref{tor free sec} we explicitly construct a group $\Gamma$ to prove the following.   

\begin{theorem} \label{tor free thm}
There exists a finitely generated, torsion-free group $\Gamma$ and homomorphism $\rho: \Gamma \to \Diff^\infty(S^1)$ that does not extend to a $C^1$ action of $\Gamma$ on $B^2$.  
\end{theorem} 

\noindent  Note that in contrast to Theorem \ref{tor free thm}, any action of $\Z$, of a free group, or any action of any group that is conjugate into the standard action of $\PSL(2,\R)$ on $S^1$ \emph{will} extend to an acton by diffeomorphisms on $B^2$.  

\item \textbf{Exotic homomorphisms.}
In Section \ref{proof sec}, we show that the failure of $\pi: \Diff^r_0(B^{n+1}) \to \Diff^r_0(S^n)$ to admit a section is due (at least in the case where $n$ is odd) to a fundamental difference between the \emph{algebraic structure} of groups of diffeomorphisms of spheres and groups of diffeomorphisms of balls.   We prove

\begin{theorem}\label{main thm}
There is no nontrivial group homomorphism $\Diff_0^\infty(S^{2k-1}) \to \Diff^1_0(B^m)$ for any $m$, $k \geq 1$.   
\end{theorem}
This generalizes a result of M. Herman in \cite{Herman}.  
Theorem \ref{main thm} also stands in contrast to situation with homeomorphisms of balls and spheres -- any continuous foliation of $B^{n+l}$ by $n$-spheres can be used to construct a continuous group homomorphism $\Homeo_0(S^n) \to \Homeo_0(B^{n+l})$.

\end{list}

\boldhead{Acknowledgements}
The author would like to thank Christian Bonatti, Danny Calegari, Benson Farb, John Franks and Amie Wilkinson for helpful conversations and their interest in this project, and Kiran Parkhe for comments and improvements to Section 4.


\section{Topological sections: known results}\label{top sec}

In order to contrast our work on group-theoretic sections with the (fundamentally different) question of topological sections, we present a brief summary of known results in the topological case.  

Let $\Diff(B^n \rel \del)$ denote the group of smooth diffeomorphisms of $B^n$ that restrict to the identity on $\del B^n = S^{n-1}$.  The projection $\Diff(B^n) \overset{\pi}\to \Diff(S^{n-1})$ is a fibration with fiber $\Diff(B^n \rel \del)$.  Hence, asking for a topological section of $\pi$ amounts to asking for a section of this bundle. 

In low dimensions ($n \leq 3$), it is known that the fiber $\Diff(B^n \rel \del)$ is contractible, so a topological section exists.  
The $n=2$ case is a well known result of Smale \cite{Smale}, and the (even more difficult) $n=3$ case is due to Hatcher \cite{Hatcher}.   Incidentally, $\Diff_0(B^1 \rel \del)$ is also contractible and this is quite elementary -- an element of $\Diff(B^2 \rel \del)$ is a nonincreasing or nondecreasing function of the closed interval, and we can explicitly define a retraction of $\Diff(B^1 \rel \del)$ to the identity via
$$r: \Diff(B^1 \rel \del) \times [0,1] \to \Diff(B^1 \rel \del)$$
$$r(f, t)(x) = t f(x) + (1-t)x.$$

Whether $\Diff(B^4 \rel \del)$ is contractible is an open question.  To the best of the author's knowledge, whether $\Diff_0(B^4) \overset{\pi}\to \Diff_0(S^{3})$ has a section is also open.   However, in higher dimensions $\Diff(B^n \rel \del)$ is not always contractible, giving a first obstruction to a section.  This is related to the existence of exotic smooth structures on spheres.

\boldhead{Exotic spheres}
Let $f \in \Diff(B^n \rel \del)$ be a diffeomorphism.  We can use $f$ to glue a copy of $B^n$ to another copy of $B^n$ along the boundary, producing a sphere $S^n_f$ with a smooth structure.  If $f$ lies in the identity component of $\Diff(B^n \rel \del)$, then $S^n_f$ will be smoothly isotopic to the standard $n$-sphere $S^n$.   If not, there is no reason that $S^n_f$ need even be diffeomorphic to $S^n$.  In fact, it follows from the pseudoisotopy theorem of Cerf (in \cite{Ce}) that for $n \geq 5$, the induced map from $\pi_0(\Diff(B^n \rel \del))$ to the group of exotic $n$-spheres is injective.   

Moreover -- and more pertinent to our discussion -- it follows from Smale's $h$-cobordism theorem that the map from $\pi_0(\Diff(B^n \rel \del))$ to exotic $n$-spheres is \emph{surjective}.  In particular, this means that in any dimension $n$ where exotic spheres exist, $\pi_0(\Diff(B^n \rel \del)) \neq 0$.   Let us now return to the fibration $\pi: \Diff(B^n) \to \Diff(S^{n-1})$ and look at the tail end of the long exact sequence in homotopy groups.  If we consider just the identity components $\Diff_0(B^n) \overset{\pi}\to \Diff_0(S^{n-1})$ we get 

$$... \to \pi_1(\Diff_0(B^n)) \to \pi_1(\Diff_0(S^{n-1})) \to \pi_0(\Diff(B^n \rel \del)) \to 0$$

Thus, whenever exotic spheres exist, the connecting homomorphism $$\pi_1(\Diff_0(S^{n-1})) \to \pi_0(\Diff(B^n \rel \del))$$ is nonzero, and so no section of the bundle exists.  

\begin{question} 
Does this bundle have a section in any dimensions $n \geq 5$ where exotic spheres do not exist?
\end{question}

We remark that for all $n \geq 5$, it is known that $\Diff(B^n \rel \del)$ has some nontrivial higher homotopy groups.  In fact the author learned from Allen Hatcher that recent work of Crowey and Schick \cite{CS} shows that $\Diff(B^n \rel \del)$ has infinitely many nonzero higher homotopy groups whenever $n \geq 7$.


\section{Group-theoretic sections} \label{group sec}

Recall from the introduction that a \emph{group-theoretic section} of $\pi$ is a (not necessarily continuous) group homomorphism $\phi: \Diff_0^r(S^{n-1}) \to \Diff_0^r(B^n)$ such that $\pi \circ \phi$ is the identity.  Recall also that, when $\Gamma$ is a group and $\rho: \Gamma \to  \Diff_0^r(S^{n-1})$ specifies an action of $\Gamma$ on $S^{n-1}$, we say that $\rho$ \emph{extends} to a $C^r$ action on $B^n$ if there is a homomorphism $\phi: \Gamma \to \Diff_0^r(B^n)$ such that $\pi \circ \phi = \rho$.  

The question of existence of group-theoretic sections for spheres and balls is completely answered by the following theorem of Ghys. 

\begin{theorem}[Ghys, \cite{Gh}] \label{ghys thm}
There is no section of $\Diff_0^1(B^{n+1}) \to \Diff_0^1(S^n)$.  \\Moreover, there is no extension of the standard embedding of $\Diff_0^\infty(S^n)$ in $\Diff_0^1(S^n)$ to a $C^1$ action of $\Diff_0^\infty(S^n)$ on $B^{n+1}$.  
\end{theorem}

We ask to what extent the failure of sections holds \emph{locally}, i.e. for finitely generated subgroups.   At one end of the spectrum, if $\Gamma$ is a free group, and $\rho: \Gamma \to \Diff^r_0(S^n)$ is any action, we can build an extension of $\rho$ by taking arbitrary $C^r$ extensions of the generators of $\rho(\Gamma)$  -- for instance, by using the collar neighborhood strategy sketched in the introduction.  There are no relations to satisfy so this defines a homomorphism and gives a $C^r$ action of $\Gamma$ on $B^{n+1}$.  

At the other end, a careful reading of Ghys' proof of Theorem \ref{ghys thm} gives the following corollary of Theorem \ref{ghys thm}.

\begin{corollary} \label{no sec cor}
For any $n$, there exists a finitely generated subgroup $\Gamma \subset \Diff_0^\infty(S^{2n-1})$ that does not extend to a subgroup of $\Diff_0^1(B^{2n})$.  
\end{corollary}
Although this follows directly from Ghys' proof of Theorem \ref{ghys thm}, we outline the argument below in order to illustrate some of Ghys' techniques.  We pay special attention to the $n=1$ case because we will use part of this construction in Section \ref{tor free sec}.  The reader will note that the argument is unique to odd-dimensional spheres, so does not answer the following question.  

\begin{question} \label{fin gen q}
Is there a finitely generated group $\Gamma$ and a homomorphism\\ $\rho: \Gamma \to \Diff^{\infty}_0(S^{2n})$ that does not extend to a $C^1$ (or even $C^r$ for some $1 <  r \leq \infty$) action on $B^{2n+1}$? 
\end{question}

\begin{proof}[Sketch proof of Corollary \ref{no sec cor}]
In the $n=1$ case, we can take $\Gamma$ to be a two-generated group as follows.   Any rotation of $S^1$ can be written as a commutator -- a nice argument for this using some hyperbolic geometry appears in Proposition 5.11 of \cite{Gh2} or Proposition 2.2 of \cite{Gh}.  So let $f$ and $g$ be such that their commutator $[f,g]$ is a finite order rotation, say a rotation of order 2.  We may even take $f$ and $g$ to be hyperbolic elements of $\PSL(2,\R) \subset \Diff_0^\infty(S^1)$.  Let $\tilde{f}$ and $\tilde{g}$ be lifts of $f$ and $g$ to diffeomorphisms of the threefold cover of $S^1$.  Since $f$ and $g$ have fixed points, we can choose $\tilde{f}$ and $\tilde{g}$ to be the (unique) lifts that have fixed points.  Then the commutator $[\tilde{f}, \tilde{g}]$ will be rotation of the threefold cover of $S^1$ by $\pi/3$.  Since the threefold cover of $S^1$ is also $S^1$, we can consider $\tilde{f}$ and $\tilde{g}$ as diffeomorphisms of $S^1$.  They then generate a subgroup $\Gamma$ of $\Diff_0^{\infty}(S^1)$ satisfying the relations 
\begin{enumerate}[i)]
\item $[\tilde{f}, \tilde{g}]^6 = 1$ and
\item $[\tilde{f}, [\tilde{f}, \tilde{g}]^2] = [\tilde{g}, [\tilde{f}, \tilde{g}]^2] = 1.$
\end{enumerate}
The second relation here comes from the fact that $[\tilde{f}, \tilde{g}]^2$ is the covering transformation.  There may, incidentally, be other relations satisfied by this group, but this is of no importance to us.  

We claim that $\Gamma$ does not extend to a subgroup of $\Diff_0(B^2)$.  To see this, we argue by contradiction.  Assume that there is a homomorphism $\phi: \Gamma \to \Diff_0^1(B^2)$ such that for any $\gamma \in \Gamma$, the restriction of $\phi(\gamma)$ to $\del B^2  = S^1$ agrees with $\gamma$.  

Let $r$ denote rotation of $S^1$ by $2\pi/3$, this is the element $[\tilde{f}, \tilde{g}]^2 \in \Gamma$ and so $\phi(r)$ is an order 3 diffeomorphism of the ball acting by rotation on the boundary.  In particular, it follows from Kerekjarto's theorem  that $\phi(r)$ is \emph{conjugate} to an order three rotation, hence has a unique interior fixed point $x$.   (A reader unfamiliar with Kerekjarto's theorem on finite order diffeomorphisms may wish to consult the very nice proof by Constantin and Kolev in \cite{CK}). 

By construction, $\tilde{f}$ and $\tilde{g}$ both commute with $r$ so $\phi(\tilde{f})$ and $\phi(\tilde{g})$ commute with $\phi(r)$, hence fix $x$.  The derivatives $D\phi(\tilde{f})_x$ and $D\phi(\tilde{g})_x$ commute with $D\phi(r)_x$ which acts as rotation by $2\pi/3$ on the tangent space.  Moreover, $[D\phi(\tilde{f})_x, D\phi(\tilde{g})_x] = D\phi(r)_x$.  But the centralizer of rotation by $2\pi/3$ in $\SL(2,\R)$ is abelian, so writing $D\phi(r)_x$ as a commutator of elements in its centralizer is impossible.   This is the desired contradiction, showing that no extension of the action of $\Gamma$ exists.  

The case for $n >1$ is similar.   We consider $S^{2n-1}$ as the unit sphere  $$\{(z_1, ..., z_n) \in \C^n\, | \, \sum \limits_{i=1}^n |z_i|^2 = 1\}.$$  The idea is to show that the finite order element $$r: (z_1, ..., z_n) \mapsto (\lambda_1 z_1, .... \lambda_n z_n)$$ where $\lambda_i$ are distinct $p^{th}$ roots of 1, can also be expressed as a product of commutators of elements $f_1, f_2, ... f_k$ that each commute with a power of $r$.  Then we can take $\Gamma$ to be the subgroup generated by the diffeomorphisms $f_i$.  Supposing again for contradiction that $\phi: \Gamma \to \Diff_0^1(B^{2n})$ is a section, one can show with an argument using Smith theory that the diffeomorphism $\phi(r) \in \Diff^1_0(B^{2n})$ has a single fixed point $x$.  It follows in a similar way to the $n=1$ case that the derivative of $\phi(r)$ at $x$ has abelian centralizer, giving a contradiction.  \\
 \end{proof}

\section{Actions of torsion free groups} \label{tor free sec}

The proof of Corollary \ref{no sec cor} relied heavily on finite order diffeomorphisms.   Ghys' proof of Theorem \ref{ghys thm} -- even in the case of even dimensional spheres -- also hinges on the clever use of finite order diffeomorphisms (and the tools that they bring: Smith theory, fixed sets, derivatives in $\SO(n)$, etc.).  Thus, we ask the following refinement of Question \ref{fin gen q}.  

\begin{question} Is there a finitely generated, torsion-free group $\Gamma$ and homomorphism $\rho: \Gamma \to \Diff^{\infty}_0(S^n)$ that does not extend to a smooth (or even $C^r$, for some $r \geq 1$) action on $B^{n+1}$? 
\end{question}

The following theorem answers this question for $n=1$.  

\begin{torfreethm}
There exists a finitely generated, torsion-free group $\Gamma$ and homomorphism $\phi: \Gamma \to \Diff^\infty(S^1)$ that does not extend to a $C^1$ action on $B^2$. 
\end{torfreethm}

Our proof modifies Ghys' construction by using a \emph{dynamical constraint} based on \emph{algebraic structure} to force a diffeomorphism to act by rotation at a fixed point.   The algebraic structure in question is the notion of \emph{distorted elements} and the constraint on dynamics follows from a powerful theorem of Franks and Handel.  We provide a brief introduction in the following few paragraphs; a reader familiar with this work may wish to skip ahead to Corollary \ref{FH cor 1} and the proof of Theorem \ref{tor free thm}.  

\boldhead{Distorted elements} 
Let $\Gamma$ be a finitely generated group, and let $S =  \{s_1,...,s_k\}$ be a symmetric generating set for $\Gamma$.  For an element $g \in \Gamma$, the \emph{word length} (or $S$-word length) of $g$ is the length of the shortest word in the letters $s_1,...,s_k$ that represents $g$.  We denote word length of $g$ by $|g|$.   

We say that $g \in \Gamma$ is \emph{distorted} provided that $g$ has infinite order and that $$\liminf \limits_{n \to \infty} \frac{|g^n|}{n} = 0.$$

Although the word length of $g^n$ depends on the choice of generating set $S$ for $\Gamma$, it is not hard to see that whether $g$ is distorted or not is independent of the choice of $S$.  

In \cite{FH}, Franks and Handel prove a theorem about the dynamics of actions of distorted elements in finitely generated subgroups of $\Diff_0(\Sigma)$, where $\Sigma$ is a closed, oriented surface.  The following theorem is a consequence of their main result.  We use the notation $\fix(g)$ for the set of points $x$ such that $g(x) = x$, and $\per(g)$ for the set of periodic points for $g$.

\begin{theorem}[Franks-Handel, \cite{FH}] \label{FH thm}
Suppose that $f$ is a distorted element in some finitely generated subgroup of $\Diff^1_0(S^2)$.  Suppose also that for the smallest $n>0$ such that $\fix(f^n) \neq \emptyset$, there are at least three points in $\fix(f^n)$.  Then $\per(f) = \fix(f^n)$.   
\end{theorem}

We can derive a corresponding statement about actions on the disc.  
\begin{corollary} \label{FH cor 1}
Suppose that $f$ is a distorted element in some finitely generated subgroup of $\Diff^1_0(B^2)$ with a periodic point on the boundary of period $k > 1$.  Then $\fix(f)$ consists of a single point.  
\end{corollary}

\begin{proof} 
Suppose $f$ is distorted in $\Gamma \subset \Diff^1_0(B^2)$.  By the Brouwer fixed point theorem, $f$ has at least one fixed point.  Since $f$ has a periodic point on the boundary $S^1$, all fixed points for $f$ lie in the interior of $B^2$.  
Double $B^2$ along the boundary to get the sphere, and double the action of $\Gamma$.  This can be smoothed to a $C^1$ action on $S^2$ using the techniques of K. Parkhe in \cite{Parkhe}.  The smoothing construction will not change the set of fixed or periodic points.  Applying Theorem \ref{FH thm} to the action on $S^2$, we conclude that the doubled action on the sphere can have at most two fixed points (since there are non-fixed periodic points), so the original action of $f$ has a single fixed point. \\ 
\end{proof}

With Corollary \ref{FH cor 1} as a tool, we are now in a position to prove Theorem \ref{tor free thm}.  

\boldhead{Proof of Theorem \ref{tor free thm}}
Recall the group $\Gamma \subset \Diff_0(S^1)$ from the proof of Corollary \ref{no sec cor}.  It is generated by two elements $\tilde{f}$ and $\tilde{g}$, satisfying the relations $[\tilde{f}, \tilde{g}]^6 = 1$ and $[\tilde{f}, [\tilde{f}, \tilde{g}]^2] = [\tilde{g}, [\tilde{f}, \tilde{g}]^2] = 1$.  
Let $\tilde{\Gamma}$ be the lift of $\Gamma$ to the universal central extension $\Diff^\infty_0(\R)$ of $\Diff^\infty_0(S^1)$.  This is a central extension of $\Gamma$ by an element $t$, satisfying the relation $t = [\tilde{f}, \tilde{g}]^2$.  Explicitly, we can realize $\tilde{\Gamma}$ as the group of all lifts of elements of $\Gamma$ to the infinite cyclic cover $\R$ of $S^1$.  Since $\Diff^\infty_0(\R)$ is torsion free, $\tilde{\Gamma}$ is as well.  Finally, let $\hat{\Gamma}$ be the HNN extension of $\tilde{\Gamma}$ obtained by adding a generator $a$ and relation $ata^{-1} = t^4$.   HNN extensions of torsion free groups are torsion free, so $\hat{\Gamma}$ is torsion free also.  

We now construct a homomorphism $\rho: \hat{\Gamma} \to \Diff^\infty_0(S^1)$ and show that it does not admit an extension $\phi: \hat{\Gamma} \to \Diff^2(B^2)$.  The homomorphism $\rho$ will not be faithful (and in fact the image $\rho(\Gamma)$ will have torsion), but this is besides the point -- the interesting part of this question is extending $\rho$ as an action of $\Gamma$.  For example, a nonfaithful action (with torsion or not) of a free group $F$ on $S^1$ always extends to the disc \emph{as an action of a free group} just by arbitrarily extending each generator. 

To define $\rho$, set $\rho(a) = id$, and for all $\gamma \in \tilde{\Gamma}$ let $\rho(\gamma)$ be the action of $\gamma$ on the quotient $\R/\Z$, i.e. the quotient action on the original circle $S^1$.   In other words, the image of $\rho$ in $\Diff^\infty_0(S^1)$ \emph{is} the group $\Gamma$ of Corollary \ref{no sec cor}.   Note that the fact that $\rho(t) = [\rho(\tilde{f}), \rho(\tilde{g})]$ is rotation by $2\pi/3$ ensures that the relation $\rho(a)\rho(t)\rho(a)^{-1} = \rho(t)^4$ is satisfied.  

We claim that this action does not extend to a $C^2$ action on the disc.  To see this, suppose for contradiction that some extension $\phi: \hat{\Gamma} \to \Diff^2_0(B^2)$ exists.  If $\phi(t)$ has finite order, then it must be rotation by $\pi/3$, and so has a unique fixed point $x$.  Now we make the same argument (verbatim!) as in the proof of Corollary \ref{no sec cor}: since $\phi(t)$ commutes with $\phi(\tilde{f})$ and $\phi(\tilde{g})$, both $\phi(\tilde{f})$ and $\phi(\tilde{g})$ fix $x$ and have derivatives at $x$ in $\SO(2)$.  This contradicts the fact that $\phi(t)$ is the commutator of $\phi(\tilde{f})$ and $\phi(\tilde{g})$. 

If instead $\phi(t)$ has infinite order, then it is a distorted element in $\phi(\hat{\Gamma})$.  We know also that the restriction of $\phi(t)$ to the boundary is rotation by $2\pi/3$.  Applying Corollary $\ref{FH cor 1}$, we conclude that $\phi(t)$ has a single fixed point $x$.  If the derivative $D\phi(t)_x$ were a nontrivial rotation of order at least 3, we could again look at derivatives at $x$ and give the same argument as in the finite order case to get a contradiction.  Thus, it remains only to show that $D\phi(t)_x$ is a rotation of order at least 3.   We show that it is rotation of order 3 exactly.

\begin{lemma} The derivative $D\phi(t)_x$ is a rotation of order 3. 
\end{lemma}

\begin{proof}  
Since $t$ is central in $\Gamma$ and since $\rho(a)\rho(t)\rho(a)^{-1}x = \rho(t)^4x = x$ implies that $\rho(a)x = x$, the whole group $\phi(\hat{\Gamma})$ fixes $x$.  Moreover, the derivatives of $\rho(t)$ and $\rho(a)$ at $x$ satisfy
$$D\phi(a)_xD\phi(t)_xD\phi(a)^{-1}_x = D\phi(t)_x^4.$$   
This relation in $\GL(2,\R)$ implies that either $D\phi(t)_x$ has a fixed tangent direction or is an order 3 rotation.  
Our strategy to show that it is order 3 is to compare the ``rotation number" of $\phi(t)$ at the fixed point and on the boundary.  

Blow up the disc $B^2$ at $x$ to get a $C^0$ action of $\hat{\Gamma}$ on the closed annulus, $A$.  The action of $\hat{\Gamma}$ on one boundary component of $A$ is the linear action on the space of tangent directions at $x$ (so $t$ either acts with a fixed point or as an order 3 rotation), and on the other boundary it is the original action on $\del B^2$ as an order 3 rotation.  

With this setup, we can apply the notion of ``linear displacement" from \cite{FH} and conclude that since $\rho(t)$ is distorted, it must act on each boundary component of $A$ with the same \emph{rotation number} and hence act as an order 3 rotation on both (See lemma 6.1 of \cite{FH}).  But instead of defining ``linear displacement" and ``rotation number" here, it will be faster to give a complete, direct proof for our special case.  The reader familiar with rotation numbers for circle homeomorphisms will see that it readily generalizes.  

Suppose for contradiction that $t$ acts on one boundary component of $A$ with a fixed point.  
Let $\tilde{A}$ denote the universal cover of $A$, identified with $\R \times [0,1]$ with covering transformation $T: (x_1, x_2) \mapsto (x_1 +1, x_2)$.  

Let $\tilde{t} \in \Homeo_0(\tilde{A})$ be the lift of the action of $t$ to $\tilde{A}$ with a fixed point on one boundary component, without loss of generality say $(x_0, 1)$ is fixed.  Then $\tilde{t}$ acts on $\R \times \{0\}$ as translation by $m + 1/3$ for some integer $m$.   Let $\tilde{a}$ be any lift of the action of $a$.  

Now $\tilde{a}(\tilde{t}\,)^n\tilde{a}^{-1}$ is a lift of $(\tilde{t}\,)^{4^n}$, so is of the form $(\tilde{t}\,)^{4^n}T^{l}$ for some $l$.   In particular, considering the displacement difference between the points $(x_0, 0)$ and $(x_0, 1)$ we have
\begin{align*}
\| \tilde{a} (\tilde{t}\,)^n \tilde{a}^{-1} (x_0,1) - \tilde{a}(\tilde{t}\,)^n\tilde{a}^{-1} (x_0,0) \| &= \|(\tilde{t}\,)^{4^n}(x_0,1) - (\tilde{t}\,)^{4^n}(x_0,0)\|\\ &= \|(x_0, 1) -(x_0 + (m+1/3)^{4n},1)\| \\ &\sim (m+1/3)^{4n}
\end{align*}
However, the distance $\|\tilde{a}(\tilde{t}\,)^n\tilde{a}^{-1} (x_0,1) - \tilde{a}(\tilde{t}\,)^n\tilde{a}^{-1} (x_0,0)\|$ grows \emph{linearly} in $n$ -- it is bounded by the maximum displacement of $\tilde{a}$ and $\tilde{t}$.  Precisely, if 
$$d = \max_{z \in \tilde{A}}\left\{  \max \{ \|\tilde{a}(z)-z\|, \|\tilde{t}(z)-z\|\} \right\}$$ 
then we have $$2(n+2)d + 1 \leq \|\tilde{a}(\tilde{t}\,)^n\tilde{a}^{-1} (x_0,1) - \tilde{a}(\tilde{t}\,)^n\tilde{a}^{-1} (x_0,0)\|$$
and this is our desired contradiction.  
\end{proof}

\begin{remark}
It is possible to modify the construction in the proof Theorem \ref{tor free thm} to avoid finite order elements.  The idea is to modify $\rho(\tilde{f})$ slightly so that the diffeomorphism $\rho(t) := [\rho(\tilde{f}), \rho(\tilde{g})]$ is the composition of an order 3 rotation $r$ with an $r$-equivariant diffeomorphism $h$ supported on a collection of small intervals in $S^1$ that is conjugate to a translation on these intervals.  We then modify $\rho(a)$ so that it is remains $r$-equivariant, but is conjugate to an expansion on the intervals of $\supp(h)$ -- i.e. so that $h$ and $\rho(a)$ act by a standard Baumslag-Solitar action on these intervals.  Done correctly, $\rho(\tilde{f})$, $\rho(\tilde{g})$ and $\rho(a)$ will be infinite order diffeomorphisms, and will generate a subgroup of $\Diff^{\infty}_0(S^1)$ satisfying the relations 
$[\rho(t), \rho(\tilde{f})] = [\rho(t), \rho(\tilde{g})] = 1$ and 
$\rho(t) \rho(a) \rho(t)^{-1} = \rho(a)^4$.   We leave the details to the reader.  
\end{remark}

\section{(Non-existence of) exotic homomorphisms}  \label{proof sec}

In \cite{Herman}, Michael Herman proved the following stronger version of Theorem \ref{ghys thm} in the case where $n=1$.  

\begin{theorem}[Herman, \cite{Herman}]  Any group homomorphism $\Diff_0^{\infty}(S^1) \to \Diff_0^1(B^2)$ is trivial.   
\end{theorem}
Herman's key tools are the deep fact that $\Diff_0^{\infty}(S^1)$ is simple, and the easy fact that $S^1$ is a finite cover of itself.  We combine some of these ideas with the techniques of Ghys in \cite{Gh} to prove a similar theorem for any odd dimensional sphere, with any group of diffeomorphisms of a ball as the target.   This is Theorem \ref{main thm} as stated in the introduction.

\begin{mainthm}
There is no nontrivial group homomorphism $\Diff_0^\infty(S^{2k-1}) \to \Diff^1_0(B^m)$ for any $m$, $k \geq 1$.    
\end{mainthm}

\begin{proof}
Let $n = 2k-1$ and identify $S^{n}$ with the unit sphere 
$$\{(z_1, ..., z_k) \in \C^n\, | \, \sum \limits_{i=1}^k |z_i|^2 = 1\}.$$  
For any prime $p$, there is a free $\Z_p$ -action on $S^{k}$ generated by the map $$f_p : (z_1, ..., z_k) \mapsto (\mu_1 z_1, ..., \mu_k z_k)$$ where $\mu_i$ are any $p^{\text{th}}$ roots of unity.    

Suppose that we have a nontrivial homomorphism $\phi: \Diff^\infty_0(S^{n}) \to \Diff_0(B^{m})$.  Since $\Diff^\infty_0(S^{n})$ is a simple group (a deep result due to Mather and Thurston, see e.g. \cite{Banyaga} for a proof), $\phi$ must be injective.   By the Brouwer fixed point theorem, $\phi(f_p)$ must fix a point.  Since $f_p$ is a finite order diffeomorphism, the set $\fix(\phi(f_p)) \subset B^m$ of fixed points of $\phi(f)$ is a submanifold of $B^m$ (one way to see this is to average a metric so that $f_p$ acts by isometries).  That $f_p$ is orientation preserving and of finite order further implies that $\fix(\phi(f_p))$ has codimension at least 2, this is because any finite order diffeomorphism $f$ is an isometry with respect to some metric, and if $f$ is nontrivial its derivative at a fixed point is a nontrivial finite order element of O($n$).  

Let $H$ be the group of isotopically trivial diffeomorphisms of $S^n / \langle f_p \rangle \cong S^n$.  We have an exact sequence
$$0 \to \Z_p \to H' \to H \to 1$$
where $H'$ is the group of all lifts of diffeomorphisms in $H$ to $f_p$-equivariant diffeomorphisms of $S^n$.   

We claim now that $\Z_p$ is the \emph{only} normal subgroup of $H'$.  To see this, suppose that $N \subset H'$ is a normal subgroup.  Then the image of $N$ in $H$ must either be trivial or all of $H$.  If the image is trivial, then either $N$ is trivial or $N = \Z_p$ and we are done.  If the image of $N$ in $H$ is all of $H$, we consider a $S^1\times ... \times S^1$ subgroup of $H$, where the $i^{\text{th}}$ $S^1$ factor is the norm 1 complex numbers mod $\mu_i$.  An element $(\lambda_1, ... \lambda_k) \in (S^1)^k /(\mu_1, ... ,\mu_k)$ acts on $S^n / \langle f_p \rangle$ by pointwise multiplication, 
$$(z_1, ... z_k) \mapsto (\lambda_1z_1, ... \lambda_k z_k).$$

Consider the extension $\Gamma$ as in the diagram below. 

\entrymodifiers={+!!<0pt,\fontdimen22\textfont2>}
\begin{displaymath}  
    \xymatrix @M=6pt { 0 \ar[r] & \Z_p  \ar[r] & H' \ar[r] & H \ar[r] &1\\
               0 \ar[r] & \Z_p  \ar[r] \ar@{=}[u] & \Gamma  \ar[r]  \ar@{^{(}->}[u] &S^1\times ... \times S^1 \ar[r] \ar@{^{(}->}[u]  &1}
\end{displaymath}

\noindent Specifically, $\Gamma$ is the group of all lifts of these actions $(z_1, ... z_k) \mapsto (\lambda_1z_1, ... \lambda_k z_k)$ to $S^n$, the $p$-fold cover of $S^n/\langle f_p \rangle$.  It may be helpful for the reader to consider the $n=1$ case, in which case we are just working with rotations of $S^1$ and their lifts to a $p$-fold cover of $S^1$.  

Note that $N \cap \Gamma$ is a normal subgroup of $\Gamma$ that projects to the full group $S^1 \times ... \times S^1$.  In particular, since $(\mu_1^{\frac{1}{p}}, ... \mu_n^{\frac{1}{p}}) \in S^1 \times ... \times S^1$, we know that some diffeormorphism $g$ of the form 
$$(z_1, ... z_k) \overset{g}\mapsto ( \mu_1^{n_1+\frac{1}{p}}z_1, ... \mu_k^{n_k+ \frac{1}{p}} z_k), \, \, \,\, n_i \in \Z$$ 
lies in $\Gamma$, hence in $H'$.  It follows that $g^p = f_p$ is a generator of $\Z_p$, so $\Z_p \subset N$.   Since $\Z_p \subset N$ and $N$ projects to $H$, it follows that $N = H'$, which is what we wanted to show.  

Having shown that $\Z_p$ is the only normal subgroup of $H'$, we can conclude that the action of $\phi(H')$ on $\fix(\phi(f_p))\subset B^m$ is either faithful, trivial, or has kernel $\Z_p$.  We already know that $\Z_p$ lies in the kernel -- this is $\phi(f_p)$ acting on its fix set -- so the action of $\phi(H')$ is not faithful.   If the action is trivial, then for $x \in \fix(\phi(f_p))$, we get a representation $D: H' \to \GL(m, \R) \subset \GL(m, \C)$ by sending a diffeomorphism $f$ to the derivative of $\phi(f)$ at $x$.  
Since $\phi(f_p)$ has nontrivial derivative at any point, and $\Z_p = \langle f_p \rangle$ is the only normal subgroup of $H'$, the representation $D$ must be faithful.   We will show this is impossible.  Indeed, it should already seem believable to the reader that $H'$ is a ``large" group and so is not linear.  Here is a short, elementary argument to make this clear.  

\begin{proof}[Proof that $D$ cannot be a faithful representation]
Since $D\phi(f_p)(x)$ has order $p$, after conjugation in $\GL(m, \C)$ we may assume it is diagonal of the form
$$
\left[\begin{array}{cccc}
\alpha_1 I_{n_1} &0 &\cdots &0 \\
0 &\alpha_2 I_{n_2}&\cdots &0  \\
\vdots & \vdots &\ddots &\vdots \\
0 &0 &\cdots &\alpha_k I_{n_k}\\ 
\end{array}\right]
$$
where $\alpha_i$ are each distinct $p^\text{th}$ roots of unity, the distinct complex eigenvalues of $D\phi(f_p)(x)$, and $I_{n_i}$ is the $n_i \times n_i$ square identity matrix.  

The centralizer of such a matrix in $\GL(m, \C)$ is the set of block diagonals of the form 

$$
\left[\begin{array}{cccc}
A_{n_1} &0 &\cdots &0 \\
0 &A_{n_2}&\cdots &0  \\
\vdots & \vdots &\ddots &\vdots \\
0 &0 &\cdots &A_{n_k}\\ 
\end{array}\right]
$$
with $A_{n_i} \in \GL(n_i, \C)$.  In other words, the centralizer is a subgroup isomorphic to $\GL(n_1, \C) \times \GL(n_2, \C) \times ... \times \GL(n_k,\C)$.   In particular, (after conjugation) we may view $H'$ as a subgroup of $\GL(n_1, \C) \times GL(n_2, \C) \times ... \times \GL(n_k,\C)$, with $f_p \in H$ a central element.  

Since $D\phi(f_p)(x)$ has order $p$, at least one eigenvalue is not 1.  Without loss of generality, assume $\alpha_1 \neq 1$.  Now consider the homomorphism $H' \to \R$ given by projecting $\GL(n_1, \C) \times \GL(n_2, \C) \times ... \times \GL(n_k,\C)$ onto the first factor -- i.e. onto $\GL(n_1, \C)$ -- and then taking the determinant.   We may assume that we chose $p > m$, so as to ensure that the image $\alpha_1^{n_1}$ of $f_p$ under this homomorphism is nontrivial.  However, we showed above that the subgroup generated by $f_p$ was the only normal subgroup of $H'$.  This means that this homomorphism to $\R$ must be faithful -- but this is impossible since $H'$ itself is nonabelian.  
\end{proof}

Thus, it remains only to deal with the case where $H'$ acts on $\fix(\phi(f_p))$ with kernel $\Z_p$.  In this case, we introduce an inductive argument.  Consider the diffeomorphism $$f_{p^2} : (z_1, ..., z_k) \mapsto (\nu_1 z_1, ..., \nu_k z_k)$$ where $\nu_i^2 = \mu_i$.  Then $f_{p^2}$ is an order $p^2$ diffeomorphism acting freely on $S^n$, commuting with $f_p$ and so an element of $H'$.   Since $f_{p^2} \notin \Z_p$, we know that $\phi(f_{p^2})$ acts nontrivially on $\fix(\phi(f_p))$.   Moreover, $\fix(\phi(f_{p^2})) \subset \fix(\phi(f_p))$, and is a nonempty submanifold of codimension at least two.  

As before, we consider a group of diffeomorphisms of a quotient of $S^n$.  Let $H_2$ be the group of isotopically trivial diffeomorphisms of $S^n / \langle f_{p^2} \rangle$.  Since $S^n / \langle f_{p^2} \rangle$ is a compact manifold, $H_2$ is a simple group.  Let $H_2'$ be the group of all lifts of elements of $H_2$ to $S^n$.  
The argument we gave above for $H$ works (essentially verbatim) to show that $\langle f_p \rangle \cong \Z_p$, and $\langle f_{p^2} \rangle \cong  \Z_{p^2}$ are the only normal subgroups of $H_2'$.  
   
Now consider the action of $H_2'$ on $\fix(\langle f_{p^2} \rangle)$.  If the action is trivial, we get as before a global fixed point and a linear representation $H_2' \to \GL(m, \R)$.  The argument using matrix centralizers above can be applied again in this case to derive a contradiction.   Otherwise, the action of $H'_2$ on $\fix(\phi(f_{p^2})$ in nontrivial.  In this case, we can proceed inductively by considering higher powers of $p$ and corresponding diffeomorphisms $f_{p^k}$.  Each time we will reduce the dimension of the fix set (a finite process) or derive a contradiction.\\
\end{proof}

\begin{remark} Note that our proof depended on the fact that $S^{2k-1}$ admits finite order diffeomorphisms that act freely, and so it does not readily generalize to odd dimensional spheres.  We conclude with a natural follow-up problem.  
\end{remark}

\begin{problem}
Describe all homomorphisms $\Diff_0^\infty(S^{2n}) \to \Diff_0^1(B^m)$.  Are any nontrivial?  
\end{problem}

\newpage

Dept. of Mathematics 

University of Chicago 

5734 University Ave. Chicago, IL 60637 

E-mail: mann@math.uchicago.edu

\end{document}